\documentclass[10pt]{amsart}
\usepackage[cp1251]{inputenc}
\usepackage[english,russian]{babel}
\usepackage{amsmath}
\usepackage{amssymb}
\usepackage{amsfonts,euscript}

\usepackage{calc}
\usepackage{mathrsfs}
\usepackage{tkz-graph}
\usetikzlibrary{decorations.markings}
\usepackage{color,pgf,tikz,pgfbaseimage,tikz-cd}
\usetikzlibrary{arrows,cd,calc}

\newtheorem{lemma}{Lemma}
\newtheorem{theorem}{Theorem}

\newtheorem{corollary}{Corollary}

\begin{document}
\renewcommand{\refname}{References}
\renewcommand{\proofname}{Proof.}
\renewcommand{\figurename}{Fig.}

\newcommand{\rr}{\mathbb{R}}
\newcommand \nn {\mathbb{N}}
\newcommand \zz {\mathbb{Z}}
\newcommand \bbc {\mathbb{C}}
\newcommand \rd {\mathbb{R}^d}

 \newcommand {\al} {\alpha}
\newcommand {\be} {\beta}
\newcommand {\da} {\delta}
\newcommand {\Da} {\Delta}
\newcommand {\ga} {\gamma}
\newcommand {\Ga} {\Gamma}
\newcommand {\la} {\lambda}
\newcommand {\La} {\Lambda}
\newcommand{\om}{\omega}
\newcommand{\Om}{\Omega}
\newcommand {\sa} {\sigma}
\newcommand {\Sa} {\Sigma}
\newcommand {\te} {\theta}
\newcommand {\fy} {\varphi}
\newcommand {\ep} {\varepsilon}
\newcommand{\e}{\varepsilon}
\newcommand{\eps}{\epsilon}

\newcommand{\VEC}{\overrightarrow}
\newcommand{\ra}{\rightarrow}
\newcommand{\IN}{{\subset}}
\newcommand{\NI}{{\supset}}
\newcommand \dd  {\partial}
\newcommand {\mmm}{{\setminus}}
\newcommand{\probel}{\vspace{.5cm}}
\newcommand{\8}{{\infty}}
\newcommand{\0}{{\varnothing}}
\newcommand{\vse}{$\blacksquare$}
\newcommand{\ov}{\overline}
\newcommand{\ia}{{I^*}}
\newcommand{\io}{{I^\infty}}

\newcommand {\bfep} {{{\bar \varepsilon}}}
\newcommand {\Dl} {\Delta}
\newcommand{\vA}{{\vec {A}}}
\newcommand{\vB}{{\vec {B}}}
\newcommand{\vF}{{\vec {F}}}
\newcommand{\vf}{{\vec {f}}}
\newcommand{\vh}{{\vec {h}}}
\newcommand{\vJ}{{\vec {J}}}
\newcommand{\vK}{{\vec {K}}}
\newcommand{\vP}{{\vec {P}}}
\newcommand{\vX}{{\vec {X}}}
\newcommand{\vY}{{\vec {Y}}}
\newcommand{\vZ}{{\vec {Z}}}
\newcommand{\vx}{{\vec {x}}}
\newcommand{\va}{{\vec {a}}}
\newcommand{\vga}{{\vec {\gamma}}}

\newcommand{\hf}{{\hat {f}}}
\newcommand{\hg}{{\hat {g}}}

\newcommand{\bj}{{\bf {j}}}
\newcommand{\bi}{{\bf {i}}}
\newcommand{\bk}{{\bf {k}}}
\newcommand{\bX}{{\bf {X}}}

\newcommand{\eB}{{\EuScript B}}
\newcommand{\wP}{{\widetilde P}}
\newcommand{\eU}{{\EuScript U}}
\newcommand{\eS}{{\EuScript S}}
\newcommand{\eP}{{\EuScript P}}
\newcommand{\eC}{{\EuScript C}}
\newcommand{\eM}{{\EuScript M}}
\newcommand{\eN}{{\EuScript N}}

\newcommand{\eT}{{\EuScript T}}
\newcommand{\eG}{{\EuScript G}}
\newcommand{\eK}{{\EuScript K}}
\newcommand{\eF}{{\EuScript F}}
\newcommand{\eZ}{{\EuScript Z}}
\newcommand{\eL}{{\EuScript L}}
\newcommand{\eD}{{\EuScript D}}
\newcommand{\E}{{\EuScript E}}
\def \diam {\mathop{\rm diam}\nolimits}
\def \fix {\mathop{\rm fix}\nolimits}
\def \Lip {\mathop{\rm Lip}\nolimits}

\newcommand{\red}{\textcolor{red}}
\newcommand{\yellow}{\textcolor{yellow}}
\newcommand{\blue}{\textcolor{blue}}
\newcommand{\green}{\textcolor{green!60!black}}

\thispagestyle{empty}

\title[On subarcs in some non-postrcritically finite dendrites]{On the set of subarcs\\ in some non-postrcritically finite dendrites}
\author{{N.V. Abrosimov, M.V. Chanchieva, A.V. Tetenov}}%
\address{Nikolay Vladimirovich Abrosimov
\newline\hphantom{iii} Regional Scientific and Educational Mathematical Center,
\newline\hphantom{iii} Tomsk State University,
\newline\hphantom{iii} pr. Lenina, 36,
\newline\hphantom{iii} 634050, Tomsk, Russia\vspace{5pt}
\newline\hphantom{iii} Sobolev Institute of Mathematics,
\newline\hphantom{iii} pr. Koptyuga, 4,
\newline\hphantom{iii} 630090, Novosibirsk, Russia\vspace{5pt}
\newline\hphantom{iii} Novosibirsk State University,
\newline\hphantom{iii} Pirogova str., 1,
\newline\hphantom{iii} 630090, Novosibirsk, Russia}%
\email{abrosimov@math.nsc.ru}%

\address{Marina Vladimirovna Chanchieva
\newline\hphantom{iii} Gorno-Altaysk State University,
\newline\hphantom{iii} Lenkina str., 1,
\newline\hphantom{iii} 649000, Gorno-Altaysk, Russia}%
\email{marinachan93@gmail.com}%

\address{Andrey Viktorovich Tetenov
\newline\hphantom{iii} Regional Scientific and Educational Mathematical Center,
\newline\hphantom{iii} Tomsk State University,
\newline\hphantom{iii} pr. Lenina, 36,
\newline\hphantom{iii} 634050, Tomsk, Russia\vspace{5pt}
\newline\hphantom{iii} Gorno-Altaysk State University,
\newline\hphantom{iii} Lenkina str., 1,
\newline\hphantom{iii} 649000, Gorno-Altaysk, Russia}%
\email{a.tetenov@gmail.com}%

\thanks{\sc Abrosimov, N.V., Chanchieva, M.V., Tetenov, A.V.,
On the set of subarcs in some non-postrcritically finite dendrites}
\thanks{\copyright \ 2018 Abrosimov N.V., Chanchieva M.V., Tetenov A.V}
\thanks{\rm This work was supported by the Ministry of Education and Science of Russia (state assignment No. 1.12877.2018/12.1)}

\sloppy

\maketitle {\small
\begin{quote}
\noindent{\sc Abstract. } We construct a family of non-PCF dendrites $K$ in a plane, such that in each of them all subarcs have the same Hausdorff dimension $s$, while the set of $s$-dimensional Hausdorff measures of subarcs connecting the given point and a self-similar Cantor subset in K is a Cantor discontinuum.\medskip

 \end{quote}
}

\section{Introduction}

Post-critically finite self-similar sets occupy significant position in the theory of self-similar sets. Their clear structure allowed to build contentful models of analysis and differential equations for such sets \cite{Kig,Str}. Their have also have very attractive geometric features: as it was proved by C.~Bandt in \cite{Bandt}, the set of dimensions of their minimal subarcs is finite. Particularly this holds for any postcritically finite self-similar dendrite $K$, whose set of cut points may be represented as  a countable union of images of arcs  $\ga_k, k=1,\ldots,n$ which are the components of attractor of some graph-directed IFS.

In this connection, much less is known about non-postcritically finite self-similar dendrites. Nevertheless, it turns out that non-PCF dendrites which satisfy one-point intersection property may have similar properties. 

In this paper, we show that in the case of non-postcritically finite dendrites such properties can also occur. We are constructing  a sufficiently wide family of systems of contraction similarities $\eS=\{S_0,S_1,S_2,S_3\}$ which are not postcritically finite, whose attractors $K$ are dendrites, lying in a triangle $\Da\subset \rr^2$ with the vertices $(0,0), (1,0), (1/2,\sqrt{3}/2)$ and for which the following properties hold.

\begin{enumerate}
\item All  subarcs $\gamma_{xy}\subset K, x\neq y$  and the set of cut points  of the dendrite $K$ have the same Hausdorff dimension $s$ (see Theorem 5, Corollary 6).
\item The set of $s$-dimensional measures $\ell_{Ox}$ of paths connecting the point $O=\textrm{Fix}(S_{0})$ with the points $x\in K\cap[0,1]$ lying on the base of the triangle $\Delta$ either is a one point set  or it is a self-similar Cantor discontinuum (see Theorem 7).
\end{enumerate}

\section{Preliminaries}

Let $\eS=\{S_1, S_2, \ldots, S_m\}$ be a system of {\rm(}injective{\rm)} contraction maps on the complete metric space $(X, d)$. A nonempty compact set $K\subset X$ is called {\em attractor} of the system $\eS$, if $K = \bigcup \limits_{i = 1}^m S_i (K)$. We also call the   subset $K \subset X$ {\em self-similar} with respect to $\eS$. Throughout the whole paper, the maps $S_i\in \eS$ are supposed to be  similarities and the set $X$ to be $\mathbb{R}^2$.

Let $I=\{1,2,\ldots,m\}$ be the set of indices, $\ia=\bigcup\limits_{n=1}^\8 I^n$ is the set of all finite $I$-tuples, or multiindices $\bj=j_1j_2\ldots j_n$. We write $S_\bj=S_{j_1j_2\ldots j_n}=S_{j_1}S_{j_2}\ldots S_{j_n}$ and for the set $A\subset X$ we denote $S_\bj(A)$ by $A_\bj$. We also denote by $G_\eS=\{S_\bj : \bj\in\ia\}$ the semigroup, generated by $\eS$.

Let $I^{\8}=\{{\tilde\al}=\al_1\al_2\ldots :\al_i\in I\}$ be {\em the index space}, and $\pi:I^{\8}\rightarrow K$ is {\em the index map}, which sends $\tilde\al$ to  the point $\bigcap\limits_{n=1}^\8 K_{\al_1\ldots\al_n}$.

By $\bi\bj$ we mean the concatenation of strings $\bi$ and $\bj$, the same $\bj\tilde\al$ is a concatenation $j_1\ldots j_n\al_1\al_2\ldots$

A non-empty compact $K \subset \mathbb{R}^2$ is called {\em attractor of the system} $\eS$, if $K = \bigcup_{i=1}^m S_i(K)$. The system $\eS$ is called {\em post-critically finite} (PCF), if the set $\{x \in K : \exists i_1,\ldots,i_n, j, l : S_{i_1\ldots i_n}(x) \in K_j \cap K_l\}$ is finite.

Let $\eC$ be the union of all  $S_i(K)\cap S_j(K)$, $i,j \in I, i\neq j$. {\em The post-critical set} $\eP$ of the system $\eS$ is the set of all $\alpha\in I^{\8}$ such that for some ${\bf j}\in I^*$, $S_ {\bf j}(\alpha)\in\eC$. In other words, $\eP= \lbrace \sigma^k(\alpha) : \alpha\in \eC,k\in \mathbb{N}\rbrace$, where the map $\sigma^k:I^{\8}\to I^{\8}$ is defined by $\sigma^k(\al_1\al_2\ldots)=\al_{k+1}\al_{k+2}\ldots$ A system $\eS$ is called {\em post-critically finite} (PCF) \cite{Kig} if its post-critical set is finite.

Let the set $K$ be connected. Then $K$ is arcwise connected and for any pair of points $x,y\in K$ we can consider a set $P_{xy}$ of paths form $x$ to $y$. Let $\be_{xy}=\inf\{\dim_H(\ga),\ga\in P_{xy}\}$. We will refer a path $\gamma\in P_{xy}$ as {\em minimal} one, if $\dim_{H}(\gamma)=\be_{xy}$. It was proved by C.~Bandt \cite{Bandt}, that the set of dimensions of all minimal paths in postcritically finite self-similar sets is finite. This gives the finiteness of the set of dimensions of all arcs in PCF dendrites.

\section{Construction}

Let $\Delta$ be the triangle on the plane $\rr^2$ with the vertices 
$$A_1=(0,0),\quad A_2=(1,0),\quad A_3=(1/2,\sqrt{3}/2).$$ 
Denote the set of vertices $\{A_1,A_2,A_3\}$ of $\Da$ by $V_\Da$.

Let  $p_1,p_2,p_3$ be such positive numbers that $p_1+p_2+p_3<1$. Define contraction similarities with fixed points at the vertices of $\Da$ as follows.
$$S_1=p_1z,\quad S_2=p_2z-p_2+1,\quad S_3=p_3z-p_3e^{i\frac{\pi}{3}}+e^{i\frac{\pi}{3}}.$$
Denote $\Delta_k=S_k(\Delta).$ Let $K'$ be the Cantor set generated by maps $ S_1, S_2, S_3$. Each point $x\in K'$ is defined by unique sequence of indices $\tilde a=a_1a_2\ldots $ where $a_k\in\{1,2,3\}$.

We now consider an equilateral triangle $\Delta_0$ such that its vertices $B_1, B_2, B_3$ satisfy the conditions
\begin{equation}\label{b1b2b3} 
B_1\in K'\cap S_{12}([A_2A_3]),\; B_2\in K'\cap S_{23}([A_1A_3]),\; B_3\in K'\cap S_{31}([A_1A_2]).
\end{equation}
Let $S_0$ be a similarity that maps the points $A_1,A_2,A_3$ to points $B_1, B_2, B_3$ respectively, and let $O=\textrm{fix}(S_{0})$ be its fixed point.

\resizebox{.95    \textwidth}{!}{{\scriptsize\begin{tikzpicture}[line cap=round,line join=round,>=stealth ,x=9.25cm,y=9.25cm]
\node[anchor=south west,inner sep=0] at (0,-.227) {
   \includegraphics[width=.45\textwidth]{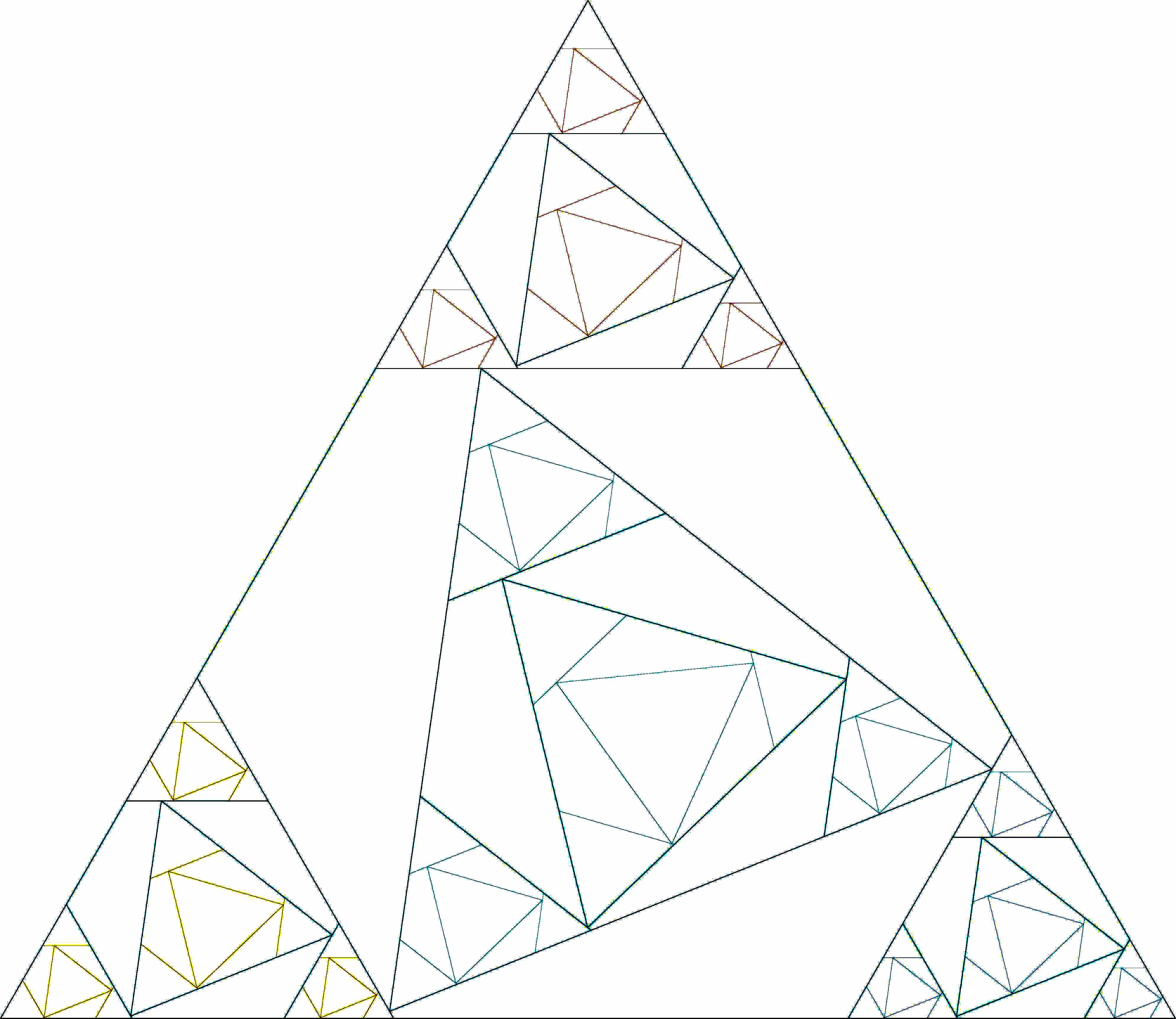}\qquad \qquad \includegraphics[width=.45\textwidth]{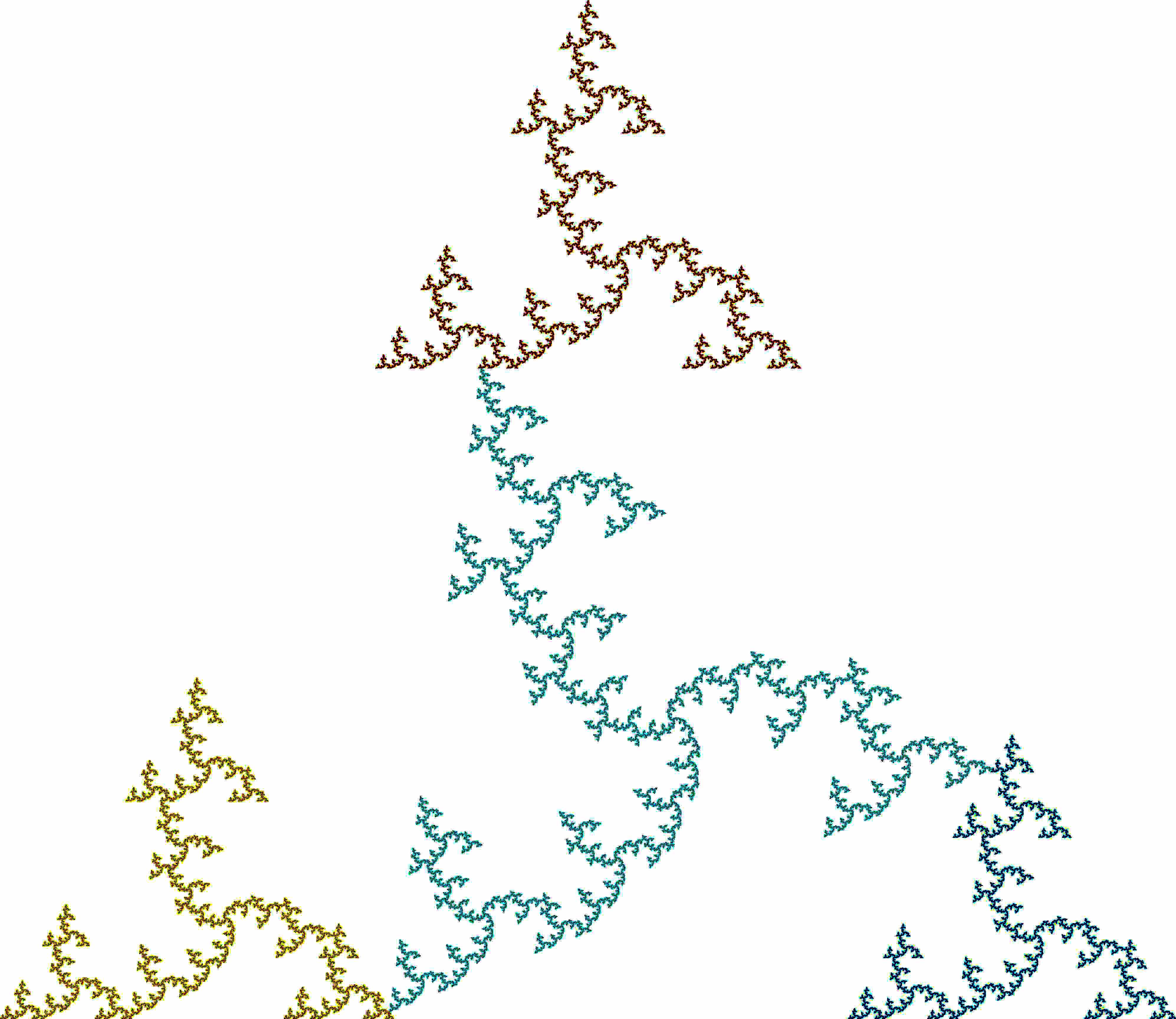}};
   \draw [line width=1 pt,color=red] (.197,.114)-- (.27,.114);
   \draw [line width=1 pt,color=red] (.207,-.226)-- (.1784,-.1755);
   \draw [line width=1 pt,color=red] (.5,-.1315)-- (.5315,-.0775); 
   \draw (0.18,-0.25) node {$S_{12}(A_2A_3)$} (0.12,0.114) node {$S_{31}(A_1A_2)$}(0.62,-0.07) node {$S_{23}(A_1A_3)$}(0.65,-0.22) node {$A_2$}(-0.025,-0.22) node {$A_1$}(0.35, 0.3) node {$A_3$}(0.23, -0.207) node {$B_1$}(0.26, 0.09) node {$B_3$}(0.49, -0.09) node {$B_2$}(0.35, -0.07) node {$O$};
\end{tikzpicture}}}

Denote by $\tilde{a}=a_1a_2\ldots, \tilde{b}=b_1b_2\ldots, \tilde{c}=c_1c_2\ldots$ the addresses of points $B_1,B_2,B_3$ in the set $K'$.
According to (\ref{b1b2b3}), $a_1=1, a_2=2;$ $b_1=2, b_2=3$ and $c_1=3, c_2=1.$

There are sufficiently many triples of points $B_1, B_2, B_3$ satisfying the conditions (\ref{b1b2b3}). For example, if $p_1=p_2=p_3$ then for any address $\widetilde{\alpha}=a_1a_2\ldots$ of point $B_1$, the points $B_2, B_3$ such addresses that $b_k=\sa(a_k), c_k=\sa(b_k)$ where $\sa=(1,2,3)$ is a cyclic permutation of indices $1,2,3$, the points $B_1, B_2, B_3$ form an equilateral triangle. The following Lemma shows that it is possible to choose the parameters $p_1,p_2,p_3,B_1,B_2,B_3$, in which the coefficients $p_1,p_2,p_3$ do not match.

\begin{lemma}
For any $\tilde c'\in\{1,2\}^\8$ different from $1111\ldots$, there exists a set of parameters $p_1,p_2,p_3,B_1,B_2,B_3$ such that $\tilde{c}=31\tilde c'$ and $p_1>p_2$.
\end{lemma}

\begin{proof}
Denote by $\eZ_{pq}$ the self-similar zipper on the interval $[0,1]$ with the vertices $\{0,p_1,1-p_2,1\}$ and the signature $(0,0,0)$. Let the map $\varphi_{p_1p_2}$ be a homeomorphism of the interval $[0,1]$ onto itself performing an isomorphism of the zipper $\eZ_{p_1p_2}$ onto the zipper $\eZ_0$ with the vertices $\{0,\frac{1}{3},\frac{2}{3},1\}$ and the same signature $(0,0,0)$.

The map $\varphi_{p_1p_2}(t)$ bijectively and continuously maps the Cantor set $K_{p_1p_2}$ generated by the maps $S_1(z)=p_1z$ and $S_2(z)=p_2z-p_2+1$ to the standard middle-third Cantor set $K_{1/3}$, and the point with the address $a_1a_2\ldots$ in $K_{pq}$ corresponds to a point with the same address in $K_{1/3}$. Denote by $\psi_{p_1p_2}(t)$ the map  inverse to $\fy_{p_1p_2}$.

Consider some pair of coefficients $p_1p_2$ and construct an equilateral triangle with vertices $S_1(\varphi_{p_2p_3}(0))$ (i.e. with the address $1222\ldots$) and $S_2(\varphi_{p_3p_1}(0))$ (i.e. with the address $2333\ldots$). The coordinates of these points are $p_1$ and $1+p_2e^{2i\pi/3}$, respectively.

From the condition of equality of the sides of the triangle $\Da_0$, we find the coordinates of the third vertex $C(p_1,p_2)=e^{i\pi/3}-p_2+p_1e^{-i\pi/3}$, hence $p_1=p_3$. Notice  that if $p_1=p_2$, the third vertex  coincides with the point $S_3(\varphi_{p_1p_2}(0))$, and as $p_2$ tends to $0$, the third vertex tends to $S_3(\varphi_{p_1p_2}(1))$.

Obviously, the coordinates of the point $C(p_1,p_2)$ continuously depend on $p_1$ and $p_2$, as does its projection $S_{3}^{-1}(C(p_1p_2))$ on the interval $[0,1]$.

The function $\varphi_{p_1p_2}(t):[0,1]\rightarrow[0,1]$ is continuous in $p_2$ and $t$ with $p_1>p_2$ and monotone in $t$ for any $p_1,p_2$. Note that $\varphi_{p_1p_2}(0)=0$ and $\varphi_{p_1p_2}(1)=1$.

Since the function $\psi_{p_1p_2} (1- \frac{p_2}{p_1} )$ is continuous in $p_2$ and vanishes at $p_1=p_2$, and at $p_2 \to 0$ it tends to $1$, then for any $a \in (0, 1)$ there is $p_2$ such that $\psi_{p_1p_2} (1- \frac{p_2}{p_1}  )=t$.

Therefore, for any point $t>0$ from $C_{\frac{1}{3}}$ with the address $c_2c_3\ldots\in\{1, 2\}^\8$ there is such $p_2<p_1$ that the triangle with the vertices $S_1(\varphi_{p_1p_2}(0)), S_2(\varphi_{p_1p_1}(0))$ and $S_3(\varphi_{p_1p_2}(t))$ is equilateral.
\end{proof}

\begin{theorem}
Let $\eS = \{S_1, S_2, S_3 , S_0\}$ be a system with parameters $p_1, p_2, p_3$, $B_1, B_2, B_3$, where each of the points $B_k$ lies in the corresponding set $S_k(\varphi_{p_jp_i}(a_k))$, where $a_k \in C_{\frac{1}{3}}$, and addresses of points $B_1, B_2, B_3$ satisfy condition (\ref{b1b2b3}). Then the attractor $K$ of the system $\eS$ is a dendrite.
\end{theorem}

\begin{proof}
Let for every $\bj\in I^*$, $\Delta_\bj=S_\bj(\Delta)$. Let $T(A)=\bigcup\limits_{k=0}^4 S_k(A)$ be a Hutchinson operator of the system $\eS$. We set $\widetilde{\Delta^k}=T^k(\Delta)=\bigcup\limits_{\bj\in I^k} \Da_\bj$.
Each system $\widetilde{\Delta^k}$ has the following properties.

\begin{enumerate}
\item For each $k$, the set $\widetilde{\Delta^k}$ is connected, locally connected, and simply connected.
\item The diameter of each $\Delta_\bi = S_\bi(\Delta)\leqslant (Lip S_0)^{k}$.
\item If $|\bi_1|=|\bi_2|=k$, and $\bi_1 \neq \bi_2$ then $\Delta_{\bi_1} \cap \Delta_{\bi_2}$ is either empty or one-point set, which is a vertex of one of the triangles $\Delta_{\bi_1}$,  $\Delta_{\bi_2}$.
\item The sets $\widetilde{\Delta^k}$ form a decreasing sequence $\widetilde{\Delta^1}\supset \widetilde{\Delta^2} \supset \ldots$
\end{enumerate}

The attractor $K$ of the system $\eS$ coincides with the intersection of the decreasing sequence of the sets $\widetilde{\Delta^k}$, and therefore, according to \cite[Lemma~2, Theorem~5]{URO}, is a dendrite.
\end{proof}

\begin{lemma}
If at least one of the addresses $\tilde a,\tilde b,\tilde c$ is not periodic, then the system $\eS$ is not postcritically finite.
\end{lemma}

\section{Dimensions and measures of subarcs}

Let $O$ be a fixed point of the map $S_0$. Denote by $\gamma_k$ the subarcs $\gamma_{OA_k}\subset K$ with endpoints $O$ and $A_k$, $k=1,2,3$. It is easy to verify that these arcs have equal dimensions.

\begin{lemma}
$dim_H(\gamma_1)=dim_H(\gamma_2)=dim_H(\gamma_3)$.
\end{lemma}

\begin{proof}
Indeed, according to the conditions (BB) we have
\begin{align*}
\gamma_1&\supset S_1S_0(\gamma_2),\\
\gamma_2&\supset S_2S_0(\gamma_3),\\
\gamma_3&\supset S_3S_0(\gamma_1).
\end{align*}
 
Therefore, $dim_H(\gamma_1)\geqslant dim_H(\gamma_2)\geqslant dim_H(\gamma_B)\geqslant dim_H(\gamma_1)$, that gives the result of the Lemma.
\end{proof}

Denote by $s$ the $dim_H(\gamma_1)$. Note that $s\geqslant 1$.

\begin{theorem}
The dimension of the set $CP(K)$ of cut points in $K$ is $s$.
\end{theorem}

\begin{proof}
Let $\bj=j_1\ldots j_k$ be the multiindex of length $k$. If the index $j_k$ is $0$ then intersection of the triangle $\Da_\bj$ with the set $\bigcup\limits_{\bi\in I^k\mmm\{\bj\}} \Da_\bi$ consists of no more than three points, each of which is a vertex of the triangle $\Da_\bj$. If $j_k\neq 0$ then for each $\bi\in I^k\mmm\{\bj\}$ for which $\Da_\bj\cap\Da_\bi\neq\0$, the index $i_k=0$, and the intersection is a one-point set consisting of a single vertex of the triangle $\Da_\bi$. Therefore, for every $\bj=j_1\ldots j_k$, the set $K_\bj\cap\overline{K\mmm K_\bj}$ consists of no more than three points, each of which lies in $G_\eS(V_\Da)$.

Let us now take some point $x\in CP(K)$ not lying in $G_\eS(V_\Da)$. This point has the single address $j_1j_2j_3\ldots\in I^\8$. Let $U_1,U_2$ be any two connected components of $K\mmm\{x\}$, and $0<d<\min(\diam U_1, \diam U_2)$. Suppose there is such $k$, that for $\bj=j_1\ldots j_k$, $j_k=0$, and $\diam \Da_\bj<d$. Then each of the intersections  $U_i\cap K_\bj\cap\overline{K\mmm K_\bj}$ is one of the vertices $A'_i$ of the triangle  $\Da_\bj$, and the arc $\ga_{A'_1A'_2}$ is a subarc $S_\bj(\hat\ga)$, containing the point $x$.

If there is no such $k$ then for some  $N$, $k>N$ implies $j_k\neq 0$ which means that $x\in G_\eS(K')$.

Thus $CP(K)\subset G_\eS(\hat\ga\cup K')$. Since $\dim \hat\ga=s>1$, and $\dim K'\le 1$ then $\dim_H CP(K)=s$.
\end{proof}

\begin{corollary}
For any subarc $\ga_{xy}\in K$, $\dim_H(\ga_{xy})=s$.
\end{corollary}

\begin{proof}
Since the set $ G_\eS(K')$ is totally disconnected then the set of all points, whose addresses contain infinite number of zeros is dense in $\ga_{xy}$. As it was shown in the previous proof, for any $z\in \ga_{xy}$, different from $x$ and $y$, there is such multiindex $\bj$, that $\ga_{xy}\cap K_\bj\subset S_\bj(\hat\ga)$. Therefore, $\dim_H(\ga_{xy})\ge s$. Since $\ga_{xy}\subset CP(K)$ then $\dim_H(\ga_{xy})= s$.
\end{proof}

Let $M$ be the set of arcs $\gamma_{Ox}$ with endpoits in $O$ and $x \in K\bigcap[0, p_1 ]$, $N$ be the set of arcs $\gamma_{Ox}$ with endpoints $O$ and $x \in K\bigcap[1-p_2, 1 ]$. Let $\eM=\{\ell_{Ox} : x \in K\bigcap[0, p_1 ]\}$ be the set $s$-dimensional Hausdorff measures of the arcs $\gamma_{Ox}\in M$ and $\eN=\{\ell_{Ox} : x \in K\bigcap[1-p_2, 1 ]\}$ be the set $s$-dimensional Hausdorff measures of the arcs $\gamma_{Ox}\in N$.

\begin{theorem}
The sets $\eM, \eN$ are the components of the attractor of some graph-directed system of contraction similarities in
$\rr$, and  
$$\dim_H\eM=\dim_H\eN\le\dim_H(K'\cap[0,1])<1.$$
\end{theorem}

\begin{proof}
First suppose, that $B_1 \neq p_1$. Then there is such $n$ that $B_1 \in S_1S_2^{n}(K)$ and $B_1\notin S_1S_2^{n+1}(K)$. Let $\lambda_k$ where $k=0,\ldots, n$ be the arc in $K$ whose endpoints are $S_1S_2^k(O)$ and $O$. Let $\lambda'$ be the arc with endpoints $O$ and $S_2(O)$. Since each arc, connecting $O$ and $S_1S_2^k(x)$ where $x\in K\cap[0,p_1]$, is a sum of subarcs $\la_k$ and $S_1S_2^k(\ga_{Ox})$ then allowing some liberty of notation we can represent the set of arcs $\{\ga_{Ox} : x\in S_1S_2^k(K\cap[0,p_1])\}$ as $\lambda_k+S_1S^k_2(M)$. The same way we get the representation of the form $\lambda'+S_2(M)$ for the set of arcs $\{\ga_{Ox} : x\in S_2 (K\cap[0,p_1])\}$, $\lambda_n+S_1S^n_2(N)$ for the set of arcs $\{\ga_{Ox} : x\in S_1S_2^n(K\cap[1-p_2,1])\}$, and $\lambda'+S_2(N)$ for the set of arcs $\{\ga_{Ox}:x\in S_2(K\cap[1-p_2,1])\}$.

Therefore we can write the following system of representations for the sets $M$ and $N$.

\begin{equation}\begin{cases} \label{eq1}
\begin{split}  M=(\lambda_0+S_1(M))\cup  (\lambda_1+&S_1S_2(M)\cup  \ldots \\ &\cup (\lambda_{n}+S_1S_2^{n}(M)) \cup (\lambda_{n}+S_1S_2^{n}(N))\end{split}\\
N=(\lambda'+S_2(M))\cup (\lambda'+S_2(N))
\end{cases}\end{equation}

We denote $\ell_k=H^s(\lambda_k)$ and $\ell'=H^s(\la')$.

Therefore, defining the contraction linear maps  
$$\sa'(x)={\ell'}+p_2x,\quad\sa_0(x)=\ell_0+p_1x,\quad\ldots\quad,\  \sa_n(x)=\ell_n+p_1p_2^nx$$ 
in $\rr$, we come to the following graph-directed system of similarities $\eS_{\eM\eN}$ for the sets $\eM\subset\rr$ and $\eN\subset\rr$.
\begin{equation}\begin{cases} \label{eq2}
\eM=\sa_0( \eM )\cup  \sa_1(\eM) \cup  \ldots \cup \sa_{n}(\eM ) \cup  \sa_{n}(\eN )\\
\eN= \sa'(\eM )\cup \sa'(\eN)
\end{cases}\end{equation}

Let us evaluate the similarity dimension $d$ of the system $\eS_{\eM\eN}$. According to \cite{MW}, $d$ is the unique value of the parameter $t$, such that for certain $\mu,\nu>0$ the following equations hold.
\begin{equation}\begin{cases} \label{eq3}
\mu=p_1^t(1+p_2^t+\ldots +p_2^{nt})\mu+ p_1^tp_2^{nt}\nu\\
\nu=p_2^t\mu+p_2^t\nu
\end{cases}\end{equation}

Expressing $\nu=\dfrac{p_2^t\mu}{1-p_2^t}$ in the second equation and substituting to the first one, we get $(1-p_2^t)\mu=p_1^t$. Thus, $d$ is the unique solution of the equation $p_1^t+p_2^t=1$. Therefore, $d\le\dim_H(K'\cap[0,1])<\dim_H(K')<1$.

\hspace{-1em}\resizebox{.95    \textwidth}{!}{
 \begin{tikzpicture}[line cap=round,line join=round,>=triangle 45,x=9.25cm,y=9.25cm]
\node[anchor=south west,inner sep=0] at (0,-.227) {
   \includegraphics[width=.85\textwidth]{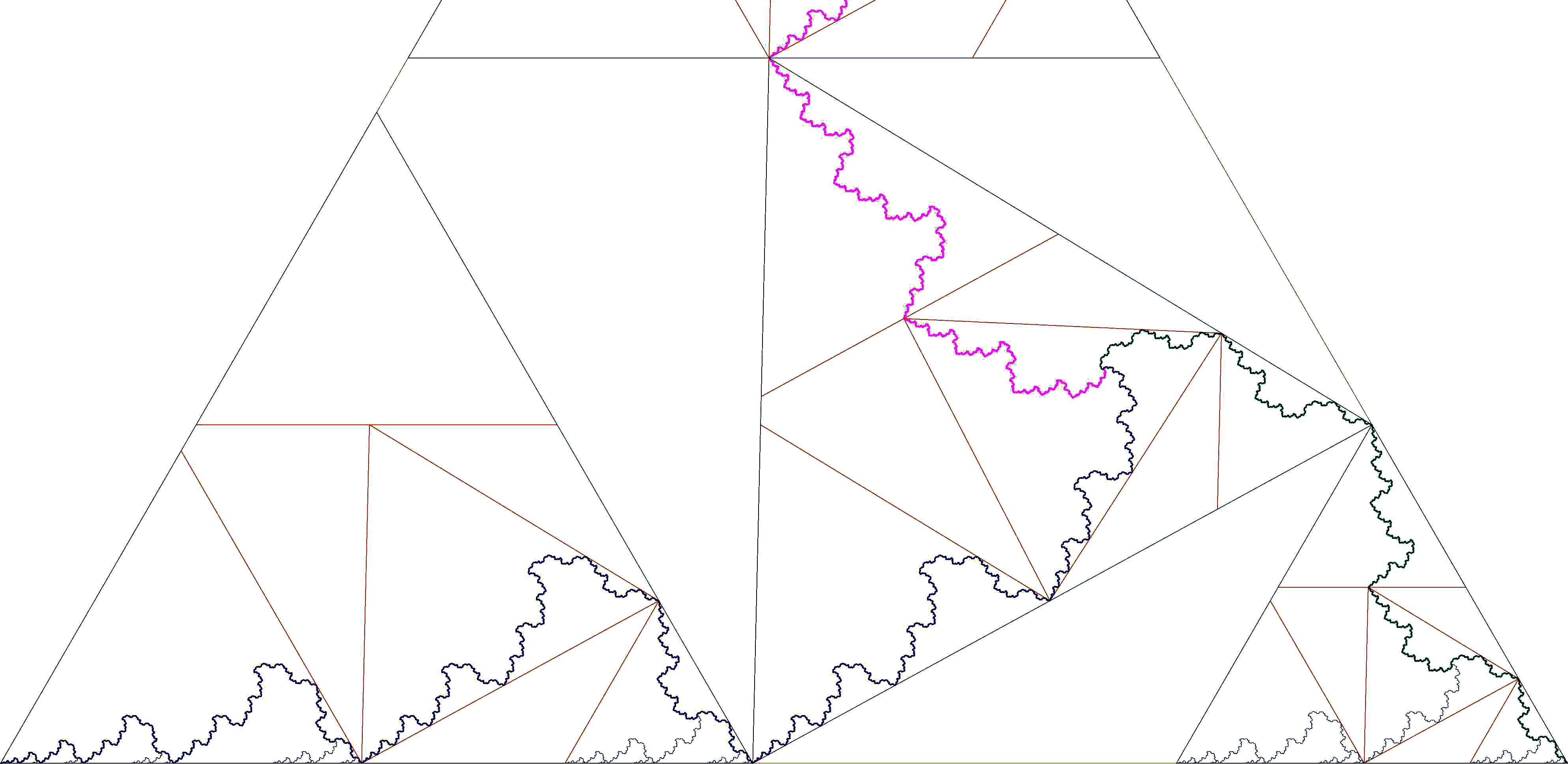}};
\draw (0.02,-0.2) node[anchor=north east] {$A_1 $};
\draw  (1.16,-0.2) node[anchor=north west] {$A_2 $};
\draw (1.02,0.05) node[anchor=north west] {$B_2$};
\draw  (0.5,0.33) node[anchor=north west] {$B_3$};
\draw  (0.76,0.05) node[anchor=south west] {$O$};
\draw  (0.52,-0.22) node[anchor=north west] {$B_1=p_1$};
\draw (.41,-.12) node [anchor=south east] {$O_1$};
\draw (1.02,-0.13) node[anchor=north west] {$O_2$};
\end{tikzpicture}}

In the case when $B_1=p_1$, denote by $M'$ the set of arcs $\ga_{A_2x}$ where $x\in K\cap[0,1]$. Notice that $M=\ga_{B_1O}+S_1(M')$, and each arc in $M'$ either belongs to $S_2(M')$ or is the sum of some subarc in $S_1(M')$ and the arc $\ga_{A_1B_1}\subset K$. Putting $\sa_1(x)=p_2x$, $\sa_2(x)=\ell_{A_2B_1}+p_1x$, $\sa_3(x)=\ell_{OS_2(B_1)}+p_1p_2x$ and $\sa_4(x)=\ell_{OO_2}+p_2x$, we get the following graph-directed system of similarities for $\eM'$ and $\eN$.
\begin{equation}\begin{cases} \label{eq4}
  \eM'=\sa_1( \eM' )\cup  \sa_2(\eM') \\
\eN= \sa_3(\eM' )\cup \sa_4(\eN )
\end{cases}\end{equation}
Thus, the set $\eM'$ is the attractor of the system $\{\sa_1,\sa_2\}$ whose similarity dimension $t$ satisfies the equation $p_1^t+p_2^t=1$.
\end{proof}

\end{document}